\documentclass[11pt,a4paper]{amsart}
\usepackage[utf8]{inputenc}
\usepackage[english]{babel}
\usepackage{amsmath}
\usepackage{amsfonts}
\usepackage{amssymb}
\usepackage{amsthm}
\usepackage{hyperref}
\theoremstyle{plain}
\newtheorem{them}{Theorem}[section]
\newtheorem{lem}[them]{Lemma}
\newtheorem*{them'}{Theorem}
\newtheorem{prop}[them]{Proposition}
\newtheorem{cor}[them]{Corollary}

\theoremstyle{definition}
\newtheorem{defi}{Definition}

\theoremstyle{remark}
\newtheorem{rem}{Remark}
\newtheorem{ex}{Example}

\author{Yanis Amirou}
\address{D\'{e}partement de math\'{e}matiques et applications, Ecole Normale Supérieure de Paris, PSL Research University, 75005 Paris, France}
\email{yanis.amirou@ens.fr}
\keywords{Word-length, word metrics, characteristic subgroups, FC-center, virtually abelian groups, torsion groups, group identities, Burnside groups, hyperbolic groups.}
\thanks{This work is partially supported by the ERC grant GroIsRan.}
\begin{document}
\title{Elements of uniformly bounded word-length in groups}

\maketitle
\begin{abstract} We study a characteristic subgroup of finitely generated groups, consisting of elements with uniform upper bound for word-lengths. For a group $G$, we denote this subgroup by $G_{bound}$. We give sufficient criteria for triviality and finiteness of $G_{bound}$. We prove that if $G$ is virtually abelian then $G_{bound}$ is finite. In contrast with numerous examples where $G_{bound}$ is trivial, we show that for every finite group $A$, there exists an infinite group $G$ with $G_{bound}=A$. This group $G$ can be chosen among torsion groups. We also  study the group $G_{bound}(d)$ of elements with uniformly bounded word-lengths for generating sets of cardinality less  than $d$.
\end{abstract}
%
\section{Introduction}
Let $G$ be a finitely generated group and $S=\{s_{1},\dots ,s_{k}\}$ be a finite generating set  of $G$. The word-length of an element $g\in G$ with respect to $S$ is denoted by $l_{S}(g)$. It is the minimal integer $n$ such that there exist $s_{i_{1}},\dots, s_{i_{n}}$ in $S\cup S^{-1}$ satisfying $g=s_{i_{1}}\dots s_{i_{n}}$.\medskip

It is clear that the word-length function $l_{S}$ depends on the generating set $S$. If $S_{1}$ and $S_{2}$ are two finite generating sets of $G$, then the word-length functions $l_{S_{1}}$ and $l_{S_{2}}$ are bi-lipschitz equivalent. This means that there exist two constants $ C_{1}$ and $C_{2}$, such that $C_{1}l_{S_{2}}\leq l_{S_{1}}\leq C_{2} l_{S_{2}}$. Given a positive function $l:G\to \mathbb{R}^{+}$, bi-lipschitz equivalent to $l_{S_1}$, one can ask when $l$ is realisable as $l_{S_{2}}$ for some finite generating set $S_{2}$ of $G$. In other words, given a graph $\Gamma$ which is bi-lipschitz equivalent to a Cayley graph $\Gamma(G,S_1)$, is $\Gamma$ a Cayley graph $\Gamma(G,S_2)$ of $G$? There are numerous obstacles for such $l$ to be realisable. For example, it has been shown in \cite[Theorem~1.1]{Koubi} that non-elementary hyperbolic groups have uniform exponential growth. It has been demonstrated in \cite[Theorem~1]{Arz} that non-elementary hyperbolic groups, groups with a finite index subgroup which surrjects onto a non-abelian free group and free Burnside groups with a large enough exponent are uniformly non-ameanable. These notions give constraints on the word-length functions which are independent of the generating sets considered. So if the function $\ell $ does not respect these constraints, it is not a word-length function.
\\%
In this paper, we show that in some groups even the value of $l$ at one single element can be an obstacle for $l$ to be a word-length function. Our goal is to study such elements.
\begin{defi}
An element $g\in G$ has \textit{uniformly bounded word-length} if there exists a constant $M>0$ such that $l_{S}(g)<M$ for every finite generating set $S$ of $G$. The set of elements of uniformly bounded word-length of $G$ is denoted by $G_{bound}$.
\end{defi}
We recall that the FC-center of a finitely generated group $G$ is the set of elements of $G$ with finite number of conjugates. We denote it by $FC(G)$.\\%
We recall that a \textit{characteristic subgroup} $H$ of a group $G$ is a subgroup such that $\phi(H)=H$ for every automorphism $\phi$ of $G$.
\\%
 It follows from the definition that the FC-center is a characteristic subgroup of $G$. It is known that every finitely generated group $G$ such that $FC(G)=G$ is virtually abelian, (see for instance \cite{Rob}, Theorem 4.32).
\\\\%
The notion of FC-center and FC-centrality are classical notions of group theory. We mention that these notions played recently an important
for the solution of Kaimanovich-Vershik conjecture by Frisch, Hartman, Tamuz and Vahidi-Ferdowsi (\cite{frisch-vahidi-tamuz}) who have shown that any non-FC-central group admits a measure with non-trivial Poisson boundary.  We also mention that following a work of Gournay \cite{Gournay} FC-central extensions have been recently used to construct new examples with Shalom's property $H_{FD}$ \cite{brieus-zheng}, \cite{anna-zheng}. FC-central extensions are used in \cite{anna-zheng}
to construct groups, where near optimal lower bound for growth function can not be obtained from the moment condition of measures with non-trivial Poisson boundary.
\\%
We show that the subgroup $G_{bound}$ that we study is a FC-central subgroup. It is well-known that a finitely generated group can have a vast variety of possible FC-centers. We mention a result of Hall that any countable abelian group is a center of some finitely generated group.  We also mention that taking FC-central extension can significantly change the geometry of a group: \cite{anna-zheng} studies extension of a lamplighter group with a much larger F\o{}lner function.
\\%
We stress however that FC-central elements (and even central elements) do not need to belong to $G_{bound}$. So the examples with non-trivial $G_{bound}$ that we study allows us to associate in a natural way a subgroup (of FC-center) that has stronger property than FC-centrality.
\\\\%
The fact that $G_{bound}$ is a characteristic subgroup of $G$ contained in $FC(G)$, see Lemma \ref{lm1}, implies that every finitely generated subgroup of $G_{bound}$ is virtually abelian see Corollary \ref{cr1}. In Section \ref{infinite girth} we give a sufficient criterion for triviality of $G_{bound}$ in terms of girth of $G$. Combining this condition with a result of Olshanskii and Sapir concerning the girth of hyperbolic groups, see Theorem 2 in \cite{Ol}, we obtain that $G_{bound}$ is trivial when $G$ is non-elementary hyperbolic. 
\\%
The main result in Section \ref{virtually abelien} is the following theorem: 
\begin{them'}
if $G $ is a finitely generated virtually abelian group, then $G_{bound}$ is finite.\end{them'}
This implies that for every finitely generated infinite group $G$, the subgroup $G_{bound}$ has infinite index, see Corollary \ref{cr2}. 
\\\\%
In Section \ref{prescription}, we prove the following. 
%
\begin{them'} 
Let $A$ be a finite group, there exists a finitely generated infinite group $G$ such that $G_{bound}=A$. This group can be chosen among torsion groups.
\end{them'}
In the last Section, we provide examples of elements with prescribed length in finitely generated groups. In the same Section, we study a generalization of $G_{bound}$ by fixing the cardinality of generating sets considered in the definition of $G_{bound}$ . 
\\\\%
It seems natural to ask whether there exists a finitely generated group $G$ such that $G_{bound}$ is infinite.
%
\section{Basic properties}\label{basic prop}

As we have mentioned, a straightforward consequence of the definition of $G_{bound}$ is:
\begin{lem}\label{lm1} 
$G_{bound}$ is a characteristic subgroup of $G$.
\end{lem}
\begin{proof}
It is clear that $e\in G_{bound}$. Consider $g\in G_{bound}$. We have $\sup_{S}l_{S}(g)=\sup_{S}l_{S}(g^{-1})$ because $l_{S}(g)=l_{S}(g^{-1})$. Therefore, if $g\in G_{bound}$, then $g^{-1}\in G_{bound}$.
\\%
Consider now $g_{1}, g_{2} \in G_{bound}$. Set $m_{1}=\sup_{S}l_{S}(g_{1})$ and $m_{2}=\sup_{S}l_{S}(g_{2})$. 
\\%
For every finite generating set $S$, we have $$l_{S}(g_{1}g_{2})\leq l_{S}(g_{1})+l_{S}(g_{2}) \leq m_{1}+m_{2}.$$ Thus, $g_{1}g_{2}\in G_{bound}$. Hence $G_{bound}$ is a subgroup of $G$.
\\%
Let us consider $g\in G_{bound}$ and a finite generating set $S=\{s_{1},\dots, s_{k} \}$ of $G$. There exists an expression $g=s_{i_{1}}\dots s_{i_{n}}$ whith $s_{i_{j}}\in S\cup S^{-1}$ for every $1\leq j\leq n$. Given an automorphism $A$ of $G$, we have $A(g)=A(s_{i_{1}}\dots s_{i_{n}})=A(s_{i_{1}})\dots A(s_{i_{n}})$  and the set $A(S)=\{A(s_{1}),\dots, A(s_{k})\}$ generates $G$. We remark that $l_{A(S)}(A(g))=l_{S}(g)$. For a finite generating set $E$ of $G$, $l_{E}(A(g))=l_{A^{-1}(E)}(g)$. We conclude that if $m=\sup_{S}l_{S}(g)$, then $l_{E}(A(g))=l_{A^{-1}(E)}(g)\leq m$. It implies that  $A(g)\in G_{bound}$ and $G_{bound}$ is a characteristic subgroup of $G$. 
\end{proof}
%
We recall that a group $G$ is residually finite if for every element $g\in G$, there exists a normal subgroup $H\lhd G$ with finite index such that $g\notin H$.
\begin{lem}
Let $G$ be a residually finite group and $g$ an element of $G$. 
There exists a constant $M>0$ such that for every finite generating set $S$, we have $  l_S(g) \leq M $, if and only if for every finite quotient $\pi: G\to Q$ of G, and a word-length function $\ell$  on $Q$, we have $\ell(\pi(g)) \leq M$. 
\end{lem}

\begin{proof} The first direction is trivial. Suppose that for every finite quotient $\pi: G\to Q$  of $G$ and every word-length function $\ell$  on $Q$, we have $\ell(\pi(g)) \leq M$ and there exists $S$ a finite generating set of $G$ such that  $l_S(g) > M$.  Consider a finite quotient in which the union of the ball of radius $M$ with respect to $l_S$ and the element ${g}$. The element $g$ is not in the ball of radius $M$ with respect to the projection of $l_S$ in $Q$. This is a contradiction.
\end{proof}

\begin{lem}\label{proj}
Let $G$ and $H$ be two finitely generated groups such that there exists a surjectiv morphism $\pi: G\to H$.\\%
(i) For every finite generating set $T$ of $H$, there exists a finite generating set $S$ of $G$ such that for every $g\in G$, we have $l_{T}(\pi(g))\leq l_{S}(g)$ and $\pi(S)\subset T\cup \{e\}$.\\%
(ii) We have $\pi(G_{bound})\subset H_{bound}$.
\end{lem}
\begin{proof}
(i) Let $A$ be a finite generating set of $G$. Set $T=\{t_{1},\dots,t_{n}\}$. For every $j\in\{1,\dots,n\}$, consider $r_{j}\in G$ such that $\pi(r_j)=t_j$, and set $R=\{r_1,\dots,r_n\}$. Observe that $\pi(R)=T$.\\%
For every $a\in A$, there exists a word $w_a(T)$ over $T\cup T^{-1}$ such that $\pi(a)=w_a(T)$. Let $w_a(R)$ be the word over  $R\cup R^{-1}$ obtained from $w_a(T)$ by replacing $t_j$ by $r_j$. Consider $q_a=w_a(R)^{-1}a$. We have $a=w_a(R)q_a$ and $\pi(a)=w_a(T)$.\\%
Set $W=\{w_a(R)\}_{a\in A}\cup {q_a}$ and $S=R\cup \{q_a\}_{a\in A}$. Every element of $G$ can be expressed as word over $A\cup A^{-1}$ and hence as a word over $W\cup W^{-1}$. We conclude that $S$ generates $G$.\\%
Consider $g\in G$, $l=l_S(g)$ and $s_1,\dots,s_l\in S\cup S^{-1}$ such that $g=s_1,\dots,s_l$.\\%
Since for every  $i\in\{1,\dots,l\}$, we have $l_T(\pi(s_i))\in\{0,1\}$, then $l_{T}(\pi(g))\leq l$. It implies that for every $g\in G$, we have $$l_{T}(\pi(g))\leq l_S(g).$$
(ii) Consider $g\in G_{bound}$. There exists a constant $M_g\geq 0$ such that for every finite generating set $A$ of $G$, we have $l_A(g)\leq M_g$. Let $T$ be a finite generating set of $H$ and $S$ as in (i). We have $$l_{T}(\pi(g))\leq l_S(g)\leq M_g.$$
Hence $\pi(g)\in H_{bound}$.

\end{proof}

\begin{cor}\label{lm2} If $G=A\times B$, then  $G_{bound}\subset A_{bound}\times B_{bound}$.
\end{cor}
\begin{proof}
Let $\pi_{1}:G\to A$ and $\pi_{2}:G\to B$ be the standard projections. It follows from Lemma \ref{proj} that $G_{bound}\subset A_{bound}\times B_{bound}$.
\end{proof}

For some groups we do not have equality $(A\times B)_{bound}=A_{bound}\times B_{bound}$ as we see in the following example.
\begin{ex}\label{Example1}
Consider  $A=\mathbb{Z}$ and $B=\mathbb{Z}/q\mathbb{Z}$ where $q$ is an integer such that $q>1$. Consider $G=A\times B$. We have $G_{bound}$ trivial and $A_{bound}\times B_{bound}=\{0\}\times \mathbb{Z}/q\mathbb{Z}$.
\\%
Indeed, since $B_{bound}=B$, the element $(0,1)$ is contained in $A_{bound}\times B_{bound}$. Let $p$ be a prime number such that $p>q+1$. It follows that $S=\{\pm (p,1),\pm (q+1,0)\}$ is a generating set of $G$. Hence 
$$(0,1)=(q+1)(p,1)-p(q+1,0).$$ 
Consequently,  we have $l_{S}(0,1)=p+q+1$. This implies that $l_{S}(0,1)$ tends to infinity when $p$ tends to infinity. Thus, $(0,1)$ is not in $G_{bound}$. It follows that $\{0\}\times B$, which equals 
$\{0\}\times B_{bound}$ is not contained in $G_{bound}$. Consequently,  $A_{bound}\times B_{bound}$ is not a subgroup of $G_{bound}$. 
\end{ex}
%

\begin{lem}\label{autom}
Let $G$ be a group generated by $n$ elements and $g\in G_{bound}$. We have $\vert Aut(G).g\vert \leq n^M$ where $M=sup_{S}l_{S}(g)$.
\end{lem}
\begin{proof}
Consider $g\in G_{bound}$ and $M=sup_{S}l_{S}(g)$. Let $S$ be a generating set of $G$ such that $\vert S \vert=n$. For every automorphism $A\in Aut(G)$, the set $A(S)$ generates $G$ and $l_{A(S)}(A(g))=l_{S}(g).$
This implies  
$$l_{S}(A(g))=l_{A^{-1}(S)}(g)\leq M.$$
We conclude that the orbit $Aut(G).g$ is include in $B_{S}(M)$ the ball of radius $M$ in $G$ with respect to $l_{S}$. We know that $\vert B_{S}(M)\vert \leq n^M$. Finally, $\vert Aut(G).g\vert \leq n^M$.

\end{proof}
We will use Lemma \ref{autom} for finite groups like in the proof of Theorem \ref{th1}. But in the case of infinite groups, Lemma \ref{autom} implies that if the Aut-orbit of an element $g$ is infinite then $g\notin G_{bound}$. Notice that the converse is false. In fact every element of $G=\mathbb{Z}$ have a finite Aut-orbit, but $G_{bound}$ is trivial, as will be shown in Lemma \ref{lm5}.
\begin{cor}\label{lm3} If $G$ is a finitely generated group, then $G_{bound} $ is a subgroup of $ FC(G)$.
\end{cor}
\begin{proof}
Since elements of $FC(G)$ are those which have finite orbit under the action of $Inn(G)$, it follows from Lemma \ref{autom} that $G_{bound}\subset FC(G)$.

\end{proof}
We mention that some finitely generated groups can have a large FC-center. For example a class of $FC$ central extensions of 
$\mathbb{Z} \wr \mathbb{Z}/2\mathbb{Z}$ is rich, 
and contains groups with arbitrily large F\o{}lner function (see Brieusel Zheng \cite{Brieussel-Zheng}, see also \cite{anna-zheng}).
\begin{cor}\label{cr1} 
Let $G$ be a finitely generated group. Every finitely generated subgroup of  $G_{bound}$ is virtually abelian.
\end{cor}
\begin{proof}
Let $H$ be a  finitely generated subgroup of $G_{bound}$. Since ($FC(G)\cap H$) $\subset FC(H)$ and $H\subset G_{bound} \subset FC(G)$, then $H\subset FC(H)$. Therefore, $H=FC(H)$. The claim that $H$ is virtually abelian follows from Theorem  4.3.1 in \cite{Rob}. This theorem asserts that a finitely generated group $G$  such that $FC(G)=G$ is virtually abelian.\\
\end{proof}
\section{Groups with infinite girth}\label{infinite girth}
Consider a group $G$ generated by a finite set $S$. We recall that a simple loop in the corresponding Cayley graph $\Gamma (G,S)$, is a non-trivial path starting and ending at $e$ without self intersections. 
The \textit{girth} of a Cayley graph $ \Gamma(G,S)$ is the minimal length of a simple loop in this graph.
A group has \textit{infinite girth} if for every fixed constant $k>0$, there is a finite generating set $S$ of $G$, such that the girth of $\Gamma( G,S)$ is greater than $k$.
\begin{lem}\label{loop}
Let $G$ be a finitely generated group with finite generating set $S$ and $\Gamma(G,S)$ the corresponding Cayley graph. Suppose $g\in G$ an element of order $n$ and $w$ a word with respect to $S$ representing $g$ with minimal length $k$. Consider $w'$ the shortest word among conjugates of $w$. The path obtained by concatenation of $n$ times $w'$ is a simple loop.
\end{lem}
\begin{proof}
Suppose that $w'=uwu^{-1}$ is the shortest word among conjugates of $w$. Assume that there exist a sub-word of $w'^{n}$ having the form $ss^{-1}$ where $s\in S$, then there exists a sub-word $w''$ of $w'$ satisfying $w'=sw''s^{-1}$. It follows that there exists a word $u$ such that $w=u^{-1}sw''s^{-1}u$ and we obtain that $w''$ is a shorter conjugate of $w$ than $w'$. It is a contradiction. We conclude that $w'^{n}$ is a simple loop.
\end{proof}

\begin{lem}\label{girth-Gb}
If a group $G$ has infinite girth, then $G_{bound}$ does not contain torsion elements.
\end{lem}
\begin{proof}
 Consider $g\in G$, such that there exists a minimal integer $n(g)>1$ satisfying $g^{n(g)}=e$. Assume that  $g\in G_{bound}$. There exists some constant $M>0$ such that for every finite generating set $S=\{s_{1},\dots,s_{m}\}$ of $G$, we have $l_{S}(g)\leq M$. Therefore, for every generating set $S$ of $G$, we have $l_{S}(g^{n(g)})\leq n(g)M$. We know that there exists a word $w$ over $S$, representing $g$, such that $w=s_{i_{1}}\dots s_{i_{k}}$ corresponding to a path starting at e and ending at $g$ in the Cayley graph $\Gamma(G,S)$. Let $w'$ be the shortest word among conjugates of $w$. The path corresponding to the concatenation of $n$ copies of the  word $w'\dots w'$, starts and ends in $e$. It is a non-trivial simple loop as shown in Lemma \ref{loop}.
Therefore, for every finite generating set, we have a loop of length bounded by $n(g)M$. It contradicts the infinite girth of $G$.
\end{proof}
\begin{cor}

If $G$ is a finitely generated torsion group with infinite girth, then $G_{bound}$ is trivial.

\end{cor}
There are numerous examples of groups satisfying Lemma \ref{girth-Gb}. It is the case for the first Grigorchuk group.
In fact, this group is a torsion group see Theorem 2.1 in  \cite{JI}, (see also Chapter 8, Theorem 17 in \cite{Del}) and has infinite girth as shown in  corollary 6.12 in \cite{Anna}.\\%
But a stronger statement holds for the first Grigorchuk group. 
We denote by $G$ the first Grigorchuk group. This group has trivial FC-center. In fact, it is a "just infinite" group, see Theorem 8.1 in \cite{JI}. It means that every normal subgroup has finite index. Since the FC-center of $G$ is a normal subgroup of $G$, if it is non-trivial, then it has finite index. This implies that $FC(G)$ is finitely generated. Since finitely generated subgroups of the FC-center are virtually abelian, we can conclude that $FC(G)$ is virtually abelian. Since finitely generated infinite abelian groups contain elements of infinite order, and since the first Grigorchuk group is a torsion group, we conclude $FC(G)$ is trivial.
\\\\%
However, there are examples with non-trivial FC-center and the Lemma \ref{girth-Gb} remains applicable. We want to prove that it is the case for hyperbolic groups.
\\\\%
We recall that a finitely generated group is hyperbolic if every geodesic triangle in its Cayley graph is $\delta$-thin. This means that for a geodesic triangle, every side is contained in the $\delta$-neighbourhood of the union of the two other sides. A hyperbolic group is said elementary if it is virtually cyclic or finite.
For properties of hyperbolic groups see \cite{Grom}, see also \cite{Gh}.
%

\begin{lem}\label{lm4} If G is a non-elementary hyperbolic group, then $FC(G)$ is finite.
\end{lem}
\begin{proof}
 We know that a subgroup of a hyperbolic group is infinite if and only if it contains a hyperbolic element (i,e element of infinite order), see for example Corollary 36, Chapter 8 in \cite{Gh}. We know also that if $h\in G$ is a hyperbolic element, we can find an integer $n_{1}\in \mathbb{N}$ and $x\in G$ such that the subgroup generated by $\{xh^{n_{1}}x^{-1},h^{n_{1}}\}$ is free of rank two, see proposition 5.5 in  \cite{Koubi}.  Since $FC(G)$ is normal in $G$, it follows that if $FC(G)$ is infinite, it contains a free subgroup of rank two. This is impossible because, by definition, every element of $FC(G)$ has finitely many conjugates.\\%
We conclude that $FC(G)$ is finite.
\end{proof}
\begin{cor}
Non-elementary hyperbolic groups have trivial $G_{bound}$.
\end{cor}
\begin{proof}
%
Since $G_{bound}\subset FC(G)$ as shown in Lemma \ref{lm3}, we conclude that $G_{bound}$ is finite. It is known that non-elementary hyperbolic groups have infinite girth see Theorem 2.6 in \cite{AKH}. Therefore, it follows from   Lemma \ref{girth-Gb} that $G_{bound}$ is trivial. 
\end{proof}
\section{Virtually abelian groups}\label{virtually abelien}
\begin{ex}
We have $\mathbb{Z}_{bound}=\{0\}$. Indeed, let $k\in \mathbb{Z}$, $k\geq 1$. Consider $m,n\geq 2$, and the generating set $S=\{nk,mnk+1\}$. Then $l_S(k)=m+1$. 
It follows that $k\notin \mathbb{Z}_{bound}$.
\end{ex}
\begin{lem} \label{lm5}
If $G$ is an infinite finitely generated abelian group,   $G_{bound}$ is trivial.
\end{lem}
\begin{proof}
Since $G$ is a finitely generated abelian group, there exist integers $\{p_{1}\leq p_{2} \leq \dots \leq p_{n}\rbrace$  such that $G$ is isomorphic to $\mathbb{Z}^{d}\times\mathbb{Z}/p_{1}\mathbb{Z}\times...\times\mathbb{Z}/p_{n}\mathbb{Z}$ with minimal $n$. Therefore, we will prove that groups of this form have trivial $G_{bound}$. We consider $G=\mathbb{Z}^{d}\times\mathbb{Z}/p_{1}\mathbb{Z}\times...\times\mathbb{Z}/p_{n}\mathbb{Z}$ and $A=\mathbb{Z}\times\mathbb{Z}/p_{1}\mathbb{Z}\times...\times\mathbb{Z}/p_{n}\mathbb{Z}$. \\%
For every $i\in \{1,\dots,n\}$, we have the standard projection $\pi_{i}:G\to \mathbb{Z}\times \mathbb{Z}/p_{i}\mathbb{Z}$ which is a surjection;. From Example \ref{Example1}, every group of the form $\mathbb{Z}\times \mathbb{Z}/n\mathbb{Z}$ for $n\in \mathbb{N}$ have trivial $G_{bound}$. Using Lemma \ref{proj}, we conclude that $G$ has trivial $G_{bound}$.
\end{proof}
Notice that for every abelian group $G$, the subgroup $G_{bound}$ is trivial and $FC(G)=G$.
\begin{them}\label{th1}
If $G $ is a finitely generated virtually abelian group, then $G_{bound}$ is finite.
\end{them}
\begin{proof}
Assume that $H=\mathbb{Z}^{d}$ is a normal subgroup of $G$, such that $G/H$ is finite. We can suppose that $H$ is a normal subgroup of $G$. We want to prove that $G_{bound}\cap H$ is trivial. If $g\in H\backslash
\{e\}$, then there exists an integer $N>1$ such that for every prime $p>N$,  $g\notin H^p$. If $p$ is coprime to
$\vert G/H \vert$, then $G/H^{p}$ is a semi-direct product of the elementary abelian group $H/H^{p}$ and $G/H$. This is a consequence of Schur-Zassenhaus Theorem, see for example Theorem 7.41 in \cite{rotman}. Hence every automorphism $x\to x^k$ of $H/H^{p}$ where $(k,p)=1$ extends to an
automorphism of $G/H^{p}$. Indeed, the morphism $F_{k}:G/H^{p}\to G/H^{p}$ such that $F_{k}(x,y)=(x^k,y)$ verifies $F_{k}(x,y)F_{k}(a,b)=(x^k (a^k)^y,yb)=F_{k}(xa^y,yb)$ where $a^{y}=y^{-1}ay$. It is clear that it is an automorphism  when $k$ is coprime to $p$. This implies that the sizes of the $Aut$-orbits of the images of $g$
in $G/H^{p}$ are unbounded if $p\to\infty$. Consider a generating set $S$ of $G$, we denote by $S'$ the projection of $S$ in $G/H^p$. Consider an element $g\in G_{bound}\cap H$, a constant $M>0$ which is a uniform bound for the length of $g$ in $G$ and $g'$ the projection of $g$ in $G/H^p$. we have 
$$l_{S'}(g')\leq l_{S}(g)\leq M.$$
Hence the length of $g'$ is uniformly bounded by $M$. Since, automorphisms of $\mathbb{Z}/p\mathbb{Z}$ are the morphisms of the form $x\to x^k$ with $k$ coprime to $p$, and $H/H^p =(\mathbb{Z}/p\mathbb{Z})^d$, for every distinct integers $k,k'<p$, the corresponding automorphisms $F_{k},F_{k'}$ are distinct. These automorphisms have no fixed points except trivial element in $H/H^p$. This implies that the cardinality of the Aut-orbit of $g'$ is at least $\dfrac{p}{\vert G/H\vert} $ in $G/H^p$ which is unbounded when $p\to \infty$. This contradicts Lemma \ref{autom}. We conclude that $H\cap G_{bound}$ is trivial and since $G/H$ is finite, $G_{bound}$ is finite.
\end{proof}

It may be interesting to explore to which groups the argument of Proposition \ref{th1} applies. We ask in this context : is $G_{bound}$ finite for any virtually nilpotent group $G$?
\begin{cor}\label{cr2}  
Let $G$ be a finitely generated group. Either $G_{bound}$ has infinite index or $G$ is finite and we have equality $G=G_{bound}$.
\end{cor}

\begin{proof}
Suppose that $G$ is infinite and $G_{bound}$ has finite index. It implies that $G_{bound}$ is infinite and finitely generated. From  Lemma \ref{lm3},  we know that in this case, $G_{bound}$ is virtually abelian. Therefore, $G$ is virtually abelian. Nevertheless, we know that if $G$ is virtually abelian, $G_{bound}$ is finite. It is a contradiction. As a consequence, $G_{bound}$ has infinite index if and only if $G$ is infinite. Finally, if $G$ is finite, then for every $ g\in G$ and every generating set $S$ of $G$, we have $l_{S}(g)\leq \# G$ and we conclude that $G_{bound}=G$.
\end{proof}
The following example gives a virtually abelien groups with non-trivial $G_{bound}$. We denote by $D_{8}$ the dihedral group with 8 elements.
\begin{ex}
The group $G=\mathbb{Z}\times D_{8}$ has non trivial $G_{bound}$.
This group is non-abelian and nilpotent of class 2. The center of $D_{8}$ is $C(D_{8})=\mathbb{Z}/2\mathbb{Z}=\{0,z\}$. Let $S$ be a finite generating set of $G$. Since $G$ is non abelian, there exist $x=(a,u),y=(b,v)\in S$ such that $\left[x,y\right]\neq e$. Since $D_{8}$ is nilpotent of class 2, $\forall u,v,w\in D_{4}$ we have $\left[\left[u,v\right],w\right]=e$. Therefore, $\forall r\in D_{8}$, we have $\left[\left[x,y\right],(0,r)\right]=(0,\left[\left[u,v\right],r\right])=(0,e)$ then, $\left[u,v\right]\in C(D_{4}) \setminus\{0\}=z$. It implies that
 $$l_{S}(0,z)=l_{S}(0,\left[u,v\right])=l_{S}(\left[x,y\right])\leq 2l_{S}(x)+2l_{S}(y)=4.$$
Therefore, $(0,z)\in G_{bound}$ and $\mathbb{Z}/2\mathbb{Z}\subset G_{bound}$.
\end{ex}
\section{Groups with prescribed finite $G_{bound}$}\label{prescription}

\begin{them} \label{th2}
Let $A$ be a finite group, there exists a finitely generated infinite group $G$ such that $G_{bound}=A$. This group can be chosen among torsion groups.
\end{them}

\begin{proof}
We consider $G=H\times A$ where $H$ is an infinite Burnside group such that $H_{bound}$ is trivial and the order of its elements is coprime to $\vert A \vert$. Our goal is to show that $G_{bound}=\{e\}\times A$. 
\\%
For example we can chose as $H$ the free Burnside group $B(2,p)=<a,b\vert x^p,~x\in F_{2}> $ such that $p$ is a prime number satisfying $p>\max\{\vert A \vert,665\}$. A result of Novikov and Adian \cite{Nov_adian} implies that this group is infinite. It was shown in \cite{Nov_adian} that the centraliser of every elements in $B(2,p)$ for large odd $p$ is cyclic, see also Theorem 19.5 in \cite{book_olshanski}. This implies that the FC-center is trivial. Hence in this case $B(2,p)_{bound}$ is trivial.  
\\%
A more general class of examples of the form $H\times A$ is as follow. Observe that any infinite Burnside group have an infinite quotient with trivial FC-center. Indeed, a result of Duguid and Mclain, see \cite{duguid_mclain_1956} and \cite{mclain_1956}, says that a finitely generated group $G$ has an infinite quotient $H$ with trivial FC-center if and only if $G$ is not virtually nilpotent. For an exposition of this result, see also \cite{ferdowsi}. It is well known that finitely generated torsion virtually nilpotent groups are finite, see for example proposition 2.19 in \cite{Mann}. 

First, let us show that $A\subset G_{bound}$. 
Let $g=(h,a)$ be an element of $G$. Let $q$ be the order of $a$ in $A$. We have $g^{p}=(h^{p}, a^{p~mod~q})=(e,a^{p~mod~q})$. There exist $\alpha \in \mathbb{N}$ and $\beta \in \mathbb{Z}$ such that $\alpha p+\beta q=1$ and $\alpha \leq q$.  It implies that $g^{\alpha p}=(e,a)$.
 \\%
Let $\pi_{A}:H\times A\longrightarrow A$ and $\pi_{H}:H\times A\longrightarrow H$ be the standard surjections. Consider a finite generating set  $S_{G}=\{(h_{1},a_{1}),\dots,(h_{n},a_{n})\}$ of $G$ and $s_{i}=(h_{i},a_{i})$. The set $\pi_{A}(S_{G})=S_{A}=\{a_{1},\dots,a_{n}\}$ is a generating set of $A$. Consider a word $w=s_{i_{1}}\dots s_{i_{n}}$ over $S_{G}$ with minimal length such that $\pi_{A}(w)=a$, i.e $\exists h\in H$ such that $w=(h,a)=(h,a_{i_{1}}\dots a_{i_{n}})$. It follows that $l_{S_{G}}(w)=l_{S_{A}}(a)$. Since $w^{\alpha p}=(e,a)$, we have
 $$l_{S_{G}}(e,a)\leq \alpha pl_{S_{G}}(w)\leq \alpha p l_{S_{A}}(a)\leq \alpha p\#(A)\leq qp\#(A)\leq p\#(A)^{2}.$$ Since this is true for every generating set $S_{G}$ of $G$, we conclude that $(e,a)\in G_{bound}$ and $\{e\}\times A< G_{bound}$.
\\%
From Corollary \ref{lm2}, we have $G_{bound}\subset H_{bound}\times A$ and $H_{bound}$ is trivial, we conclude that $G_{bound}=A$.

\end{proof}

\begin{rem} 
Notice the following:
\begin{enumerate}
\item In the group $G$ considered in Theorem \ref{th2},  the finite subgroup $A$ is contained in the ball of radius $p\#(A)^{2}$ of every Cayley graph $\Gamma(G,S)$ where $S$ is a finite generating set of $G$.
\item $H$ and $G$ are commensurable but $H_{bound}=\{e\}$ and $G_{bound}=A$ which is non-trivial.
\end{enumerate}
\end{rem}
\section{Remarks}
\subsection{Elements with prescribed length}\label{prescribed length}
Consider a finitely generated group $G$, an element $g\in G\setminus G_{bound}$ and an integer $k>1$. We want to answer the following question: can we construct a generating set $S$ of $G$, such that $l_{S}(g)=k$?
\\\\%
For example: $G=\mathbb{Z}^{2}$. Consider $g=(1,0)$. Given an integer $k\geq 2$ and  $S=\{(\pm1,k-1),(0,\pm1)\}$,  we have $$l_{S}(g)=l_{S}((1,k-1)+(k-1)(0,-1))=k.$$
The answer is positive for non-trivial elements of free groups.
\begin{ex}\label{presfree} Let $g\in F_{k}$ be a non trivial element and an integer $l$. There exists a generating set $E$ of  $F_{k}$ such that $l_{E}(g)=l+1$.
\\%
In fact, consider $g\in F_{k}$, $l\in \mathbb{N}\setminus \{0\}$, $p=2l+1$, $S=\{x_{1},\dots,x_{k}\}$ the standard free generating set of $F_{k}$ and $u,v$ two distinct prime numbers, such that $v>u>p$. 
Set $A=\{x_{i}^{u},x_{i}^{v}\}_{1\leq i\leq k}$ and $E=\{g^{2},g^{p}\}\cup A$. 
This is clearly a generating set of $F_{k}$ because $\{x_{i}^{u},x_{i}^{v}\}_{1\leq i\leq k}$ generates $S$. 
Let us prove that $g=g^{p-2l}$ is the shortest expression of $g$ with respect to $E$. Since $S$ is a free generating set of $F_{k}$, if an element $x_{j}$ is used for the shortest expression of $g$ with respect to $E$, then there exists an expression of $g$ with respect to $A$ such that $l_{A}(g)<l+1$.
 This means that  $$g=x_{i_{1}}^{n_{i_{1}}}\dots x_{i_{m}}^{n_{i_{m}}} $$ where $n_{i_{j}}\in\mathbb{Z}$. Notice that this expression is unique. 
 It follows that there exists $\alpha_{j},\beta_{j}\in \mathbb{Z}$ such that $\alpha_{j}u+\beta_{j}v=n_{i_{j}}$ and minimal $\vert \alpha_{j}\vert +\vert\beta_{j}\vert$. 
 Hence the shortest expression of $g$ with respect to $E$ is $$g=x_{i_{1}}^{\alpha_{1}u+\beta_{1}v}\dots x_{i_{m}}^{\alpha_{m}u+\beta_{m}v} .$$
It follows that $l_{E}(g)=\sum_{j=1}^{m} \vert \alpha_{j}\vert +\vert\beta_{j}\vert $. We have $\vert \alpha_{j} \vert=\dfrac{\vert n_{_{i_{j}}}-\beta_{j} v \vert}{u}$. When $v>3n_{i_{j}}u$, $l_{A}(g)>l+1$. In this case, $g=g^{p-2l}$ is the shortest expression of $g$ with respect to $E$ and $l_{E}(g)=l+1$.
\end{ex}
We can use the same argument in order to prove that the length of  elements of free abelien groups $\mathbb{Z}^{d}$ can be prescribed. 

\begin{ex}
Let $g\in \mathbb{Z}^{d}$ be a non trivial element and an integer $l$. There exists a generating set $E$ of  $ \mathbb{Z}^{d}$ such that $l_{E}(g)=l+1$.
\\%
In fact, consider $S=\{ z_{1},\dots,  z_{n}\}$ a free abelian generating set of $\mathbb{Z}^{d}$. Consider $A=\{uz_{i},vz_{i}\}_{1\leq i\leq k}$ and $E=\{2g,pg\}\cup A$ where $p=l+1$ and $u,v$ two prime numbers such that $v>u>p$. using the same argument as in the precedent example, we obtain that the shortest word representing $g$ with respect to $E$ is $g=(p-2l)g$. It follows that $l_{E}(g)=l+1$.
\end{ex}

\subsection{Generalization: $G_{bound}(d)$}\label{generalization}
We denote by $m(G)$ the minimal number of elements contained in a symmetric generating set of $G$. 
\begin{defi} Let $G$ be a finitely generated group.\\%
$G_{bound}(d)=\{g\in G \vert~\exists~M>0$ such that  for every finite generating set $S$ of $G$ with cardinality less or equal to $d$, $l_{S}(g)\leq M\}$.
\end{defi}
It is clear that for $d\geq m(G)$, $$G_{bound} \subset G_{bound}(d)\subset G_{bound}(d-1)\dots \subset G_{bound}(m(G))$$ and $G_{bound}=\cap_{d\geq m(G)} G_{bound}(d)$.
%
\subsection{Properties}
%
\begin{lem}\label{lm6} For every $d\geq m(G)$, the following properties of $G_{bound}$ remain true for $G_{bound}(d)$: \\%
\begin{enumerate}
\item It is a characteristic subgroup of $G$.
\item It is contained in $FC(G)$.
\item Every finitely generated subgroup of $G_{bound}(d)$ is virtually abelian.
\end{enumerate}
\end{lem}
\begin{proof}
Consider an element $g \in G_{bound}(d)$. There exists a constant $m>0$ such that for every generating set of cardinality less or equal to $d$,  $l_{S}(g)\leq m$. Given an automorphism $A$ of $G$, and a generating set $E$ of $G$ we have seen in Lemma \ref{lm1} that $l_{E}(A(g))=l_{A^{-1}(E)}(g)\leq m$. The generating sets $A^{-1}(E)$ and $E$ have the same cardinality, in particular this is true if the automorphism $A$ is a conjugation. Using the same argument as in Lemma \ref{lm2}, it follows that for every $d\geq m(G)$, we have $G_{bound}(d)$ is a subgroup of $FC(G)$. As we showed in Corollary \ref{cr1} and Lemma \ref{lm4}, property (3) is a direct consequence of property (2). 
\end{proof}
\begin{cor} We have the following:
\begin{enumerate}
\item Every non-virtually cyclic (resp non-cyclic torsion free) hyperbolic group $G$ has trivial $G_{bound}(d)$ for every $d\geq m(G)+1$ (resp $d\geq m(G)$).
\item If $G$ is the first Grigorchuk group, then for every $d\geq m(g)$, $G_{bound}(d)$ is trivial.  
\end{enumerate}
\end{cor}
\begin{proof} 

(1) Consider a non-virtualy cyclic (resp non-cyclic torsion free) hyperbolic group $G$. A consequence of Theorem 2 in \cite{Ol} is that for every $d\geq m(G)+1$ (resp $d\geq  m(G)$), $G$ has a sequence of generating sets of cardinality $d$, $\{S_{i}(d)\}_{i\in \mathbb{N}}$, such that the length of minimal loop in the Cayley graph $\Gamma (G,S_{i}(d))$ is greater than $i$. We use the same argument as in Lemma \ref{girth-Gb} to conclude that $G_{bound}(d)$ can not contain torsion elements. Using the fact that the FC-center of a hyperbolic groups is finite, and the fact that $G_{bound}(d)\subset FC(G)$, we obtain that $G_{bound}(d)$ is trivial.\\%
(2) Let $G$ be the first Grigorchuk group. It is known that $G$ is "just infinite" see Theorem 8.1 in \cite{JI}. Therefore,
if $FC(G)$ is non-trivial, then it has finite index. It follows that $FC(G)$ is finitely generated and virtually abelian, see Theorem 4.3.1 in \cite{Rob}. As a consequence, $G$ is a virtually abelian group. Since $G$ is a finitely generated torsion group, it can not be an infinite virtually abelian group. It is a contradiction. We conclude that $FC(G)$ is trivial. The fact that for every $d\geq m(G)$, we have $G_{bound}(d)\subset FC(G)$, implies that $G_{bound}(d)$ is trivial.

\end{proof}
\subsection{Examples}
\begin{prop}\label{prop1}

Consider $G=\mathbb{Z}^{d}$. 
\begin{enumerate}
\item For $d=1$, if  $m=2$, $G_{bound}(m)=G$ else $G_{bound}(m)=\{0\}$.
\item For $d\geq 2$ and $m\geq 2d$, $G_{bound}(m)=\{0\}$.
\end{enumerate}
\end{prop}
%
\begin{proof}
\hspace{0.1cm}
(1)  Consider $d=1$ and $G=\mathbb{Z}$. The generating set of $\mathbb{Z}$ with cardinality 2 is $S=\{\pm 1\}$. In this case, the length of an integer $n\in \mathbb{Z}\setminus \{0\}$ is $l_{S}(n)=n$ and $G_{bound}(2)=\mathbb{Z}$. \\%
Let $p,q$ be two different prime numbers. There exist $a,b\in \mathbb{Z}$ such that  $ap+bq=n$. For every integers $K>\vert n\vert$,  $p>K+1$ and $q>p!$, we have:
$$\vert a\vert= \dfrac{\vert bq-n\vert}{p}\geq \vert \dfrac{bp!}{p}-\dfrac{n}{p} \vert\geq \vert b(p-1)!-1\vert \geq K!-1\geq K.$$ 
We know that  $S=\{\pm p,\pm q\}$ is a generating set of $\mathbb{Z}$ and $l_{S}(1)=\vert a\vert +\vert b \vert$.  From this, we conclude that for every integers $n\in \mathbb{Z}$ and $K>\vert n\vert$, there exists a generating set $S$ such that $l_{S}(n)>K$. It implies that for $m\geq 4$, $G_{bound}(m)=\{0\}$.

(2) Let $X=(x_{1},\dots ,x_{d})\in \mathbb{Z}^{d}$ non zero i.e there exists $i\in \{1,\dots d\}$ such that $x_{i}\neq 0$.  Using an element of   $SL(d,\mathbb{Z})$, we can  suppose $x_{1}\geq 1$.   Consider $p,q$ two distinct primes, and $a,b\in \mathbb{Z}$ such that $bp-aq=1$. We have  $A=\begin{pmatrix}
p & q \\
a & b
\end{pmatrix}\in SL(2,\mathbb{Z})$ and the map $v\in \mathbb{Z}^{2}\longrightarrow Av\in \mathbb{Z}^{2}$ is an isomorphism. Therefore, there exist $\alpha, \beta \in \mathbb{Z}$ such that $A\begin{pmatrix}
\alpha  \\
\beta
\end{pmatrix}=\begin{pmatrix}
x_{1} \\
x_{2}
\end{pmatrix}$. It implies that $\alpha p+\beta q=x_{1}$.
\\%
For fixed integers $K>1$,  $p>max\{\vert x_{1}\vert,K+1\}$ and $q>p!$, we have $$\vert \alpha\vert= \dfrac{\vert \beta q-x_{1}\vert}{p}\geq \vert \dfrac{\beta p!}{p}-\dfrac{x_{1}}{p} \vert\geq \vert\beta\vert (p-1)!-1\geq K!-1\geq K.$$
Consider the generating set  of $\mathbb{Z}^{d}$ with $2d$ elements: 
$S_{(p,q)}=\{\pm v_{1},\dots, \pm v_{d}\}$ where  $ v_{1}=(p,a,0\dots 0), v_{2}=(q,b,0\dots 0), v_{3}=(0,0,1,0,\dots 0),
\dots,\\ v_{d}=(0,0,\dots ,1).$

With respect to generating set $S_{(p,q)}$,  we have $$X=\alpha v_{1}+\beta v_{2}+ \sum_{j=3}^{j=d} x_{j}v_{j}.$$ For every $K>1$,  consider $(p,q)$ as above. We have 
$$l_{S_{(p,q)}}(X)=\vert \alpha \vert +\vert \beta \vert + \sum_{j=3}^{j=d}\vert x_{j}\vert \geq K.$$
 This implies that for $d\geq 2$ and $m\geq 2d$, $G_{bound}(m)=\{0\}$.

\end{proof}
\begin{prop}\label{prop2}
Let $G=<a,b,c\vert \left[a,b\right]=c, \left[a,c\right]=\left[c,b\right]=1>$ be Heisenberg's group. We have:
\begin{enumerate}
 \item $G_{bound}(4)=\mathbb{Z}=Z(G)$ the center of $G$.
 \item For $d\geq 6$, $G_{bound}(d)$ is trivial.
\end{enumerate}
\end{prop}
\begin{proof}
(1) Let  $S=\{x^{\pm 1},y^{\pm 1}\}$ be a generating set of $G$. Set $z=[x,y]$.  Since $G$ is nilpotent of degree 2, $\forall g\in G,$ we have  $[[x,y],g ]=[z,g]=e$. This implies that $z\in Z(G)=\mathbb{Z}$.
\\%
We know that $G/Z(G)=\mathbb{Z}^{2}$.  It induces a surjective homomorphism:
$$\pi :G\longrightarrow G/Z(G)=\mathbb{Z}^{2}.$$
We want to show that $Z(G)$ is generated by $z=[x,y]$.\\%
It is known that $Z(G)$ is generated by $c=[a,b]$ with the presentation above. We have $z\in Z(G)=\mathbb{Z}$ and $c$ a generator of $Z(G)$. It follows that there exists an integer $i$ such that $z=c^{i}$.\\%
Let us show that $g=x^{\alpha}y^{\beta}$ with $\alpha \beta\neq 0$, is not an element of $Z(G)$.\\%
The set $\pi\{x^{\pm 1},y^{\pm 1}\}$ is a generating set of $G/Z(G)=\mathbb{Z}^{2}$. For  $(\alpha,\beta)\neq (0,0)$, $\pi(g)=\alpha \pi(x)+\beta \pi(y)\neq 0$. It implies that $g$ is not an element of $Z(G)$. 
Therefore, for every element $g\in G$, there exist $\alpha,\beta,\gamma \in \mathbb{Z}$ such that $g=x^{\alpha}y^{\beta}c^{\gamma}$.\\%
We can write $a=x^{\alpha_{1}}y^{\beta_{1}}c^{\gamma_{1}}$ and $b=x^{\alpha_{2}}y^{\beta_{2}}c^{\gamma_{2}}$. It follows,
\begin{center}
\begin{equation}\nonumber
\begin{array}{lll}
[a,b]&=&x^{\alpha_{1}}y^{\beta_{1}}x^{\alpha_{2}}y^{\beta_{2}}y^{-\beta_{1}}x^{-\alpha_{1}}y^{-\beta_{2}}x^{-\alpha_{2}}\\
&=&(x^{\alpha_{1}}y^{\beta_{1}}x^{-\alpha_{1}}y^{-\beta_{1}})(y^{\beta_{1}}x^{\alpha_{1}}x^{\alpha_{2}}y^{-\beta_{1}}x^{-\alpha_{1}}x^{-\alpha_{2}})\\
& &x^{\alpha_{2}}x^{\alpha_{1}}y^{\beta_{1}}y^{\beta_{2}}y^{-\beta_{1}}x^{-\alpha_{1}}y^{-\beta_{2}}x^{-\alpha_{2}}
\end{array}
\end{equation}
\end{center}
%
Since $G$ is nilpotent of degree 2, for every $x,y\in G$ and $\alpha,\beta \in \mathbb{Z}$, we have $[x^{\alpha},y^{\beta}]=x^{\alpha}y^{\beta}x^{-\alpha}y^{-\beta}=[x,y]^{\alpha \beta}$. \\%
It follows that:
\begin{equation}\nonumber
\begin{array}{lll}
x^{\alpha_{1}}y^{\beta_{1}}x^{-\alpha_{1}}y^{-\beta_{1}}&=&z^{u_{1}},\\
y^{\beta_{1}}x^{\alpha_{1}}x^{\alpha_{2}}y^{-\beta_{1}}x^{-\alpha_{1}}x^{-\alpha_{2}}&=&z^{u_{2}}.\\
x^{\alpha_{1}}y^{\beta_{2}}x^{-\alpha_{1}}y^{-\beta_{2}}&=&z^{u_{3}}\\
\end{array}
\end{equation}
Where $u_{1}=\alpha_{1}\beta_{1}$, $u_{2}=-\beta_{1}(\alpha_{1}+\alpha_{2})$ and $u_{3}=\alpha_{1}\beta_{2}$.
Consider $j= u_{1}u_{2}u_{3} $. Using these equations in the expression of $[a,b]$, we obtain the following:\\
\begin{equation}\nonumber
\begin{array}{lll}

[a,b]&=&z^{u_{1}u_{2}}x^{\alpha_{2}}(x^{\alpha_{1}}y^{\beta_{2}}x^{-\alpha_{1}}y^{-\beta_{2}})x^{-\alpha_{2}}\\
&=&z^{u_{1}u_{2}u_{3}}=z^{j}
\end{array}
\end{equation}
Accordingly, $c^{ij}=c$. This implies that $\vert i\vert=1$ and $z$ generates $Z(G)$. It follows that for every generating set $S$ of $G$ with $4$ elements, $l_{S}(c)\leq 4$. As a result $Z(G)\subset G_{bound}(4) $.
We recall that $G_{bound}(4)\subset FC(G)$. We will show that $FC(G)=Z(G)$.\\%
 It is known that every element of $G$ has a unique expression of the form $g=a^{i}b^{j}c^{l}$ where $i,j,l\in \mathbb{Z}$. Therefore,\\%
if $j\neq 0$, then $\forall n\in \mathbb{Z}$, 
 $$a^{n}(a^{i}b^{j}c^{l})a^{-n}=a^{i}(a^{n}b^{j}a^{-n})c^{l}.$$ Since $a^{n}b^{j}a^{-n}=c^{nj}b^{j}$, it follows that 

$$a^{n}(a^{i}b^{j}c^{l})a^{-n}=a^{i}(c^{nj}b^{j})c^{l}=a^{i}b^{j}c^{l+nj}.$$
If  $i\neq 0$, then $\forall n\in \mathbb{Z}$ we have $$b^{n}(a^{i}b^{j}c^{l})b^{-n}=a^{i}b^{j}c^{l-ni}.$$
Consider $h=a^{n}b^{m}c^{k}$ a non-trivial element. We have 
\begin{equation}\nonumber
\begin{array}{lll}

hgh^{-1}&=&a^{n}b^{m} a^{i}b^{j}c^{l} b^{-m}a^{n}\\
\quad \quad &=&a^{n}(b^{m}( a^{i}b^{j}c^{l}) b^{-m})a^{n}\\
\quad \quad &=&a^{n}( a^{i}b^{j}c^{l-mi})a^{-n}\\
\quad \quad &=&a^{i}b^{j}c^{l-mi+nj}
   
\end{array}
\end{equation}
%
It implies that elements which are not in $Z(G)$ have an infinite number of conjugates. So, they are not in $FC(G)$. We conclude  that $FC(G)=Z(G)$.\\\\%
(2) We know that $G_{bound}(6)<FC(G)$. Consider  two prime number $p,q$.  The set $S=\{a^{\pm p},a^{\pm q},b^{\pm 1}\}$ is a generating set of $G$. Let us prove that for every $n\in \mathbb{Z}$, $c^{n}\not\in G_{bound}(6)$. For every $u,v\in \mathbb{Z}$ such that $uv=n$ we have $c^{n}=[a^{u},b^{v}]$. There exist $\alpha_{u}, \beta_{u} \in \mathbb{Z}$ such that $\alpha_{u} p+\beta_{u} q=u$. 
It follows that $c^{n}=[a^{u},b^{v}]=a^{\alpha_{u} p+\beta_{u} q}b^{v}a^{-\alpha_{u} p-\beta_{u} q}b^{-v}$. Consider $A$ the set of divisors of $n$. This set is finite. We have $l_{S}(c^{n})\geq min \{\vert \alpha_{u}\vert +\vert \beta_{u} \vert ~\vert~u\in A\}$ which tends to infinity when $p$ tends to infinity. 
 Since for every $d\geq 6$, $G_{bound}(d)\subset G_{bound}(6)$, we conclude that $G_{bound}(d)$ is trivial for $d\geq 6$.

\end{proof}
\begin{cor}\label{heis}
The Heisenberg group has trivial $G_{bound}$.
\end{cor}
\begin{proof}
It follows from $G_{bound}=\cap_{d\geq m(G)} G_{bound}(d)$ and $G_{bound}(6)$ is trivial.
\end{proof}
\section*{Acknowledgements}
I would like to express my deep gratitude to my supervisor, Anna Erschler, for the advices and help she gives me in my work. I am grateful to Rostislav Grigorchuk, Pierre de La Harpe, Tatiana Smirnova-Nagnibeda and Markus Steenbock for remarks on the preliminary version and discussions that have improved the exposition. I thank Arman Darbinyan for his comments.

\bibliographystyle{plain}
\bibliography{article}

\end{document}